\documentclass[11pt]{article}

\usepackage[a4paper,margin=30mm]{geometry}
\usepackage[T1]{fontenc}
\usepackage{lmodern}
\usepackage{microtype}
\usepackage{amsmath,amssymb,amsthm,mathtools}
\usepackage{booktabs}
\usepackage{enumitem}
\usepackage[hidelinks]{hyperref}
\hypersetup{
  pdftitle={An Almost-Covering Threshold for Golomb-Ruler Difference Packings},
  pdfauthor={Chaohang Ma and Xiangjie Yi},
  pdfkeywords={Golomb ruler, perfect difference family, difference triangle set, Sidon set, Fourier obstruction}
}

\newtheorem{theorem}{Theorem}[section]
\newtheorem{proposition}[theorem]{Proposition}
\newtheorem{lemma}[theorem]{Lemma}
\newtheorem{corollary}[theorem]{Corollary}
\theoremstyle{remark}

\newcommand{\Z}{\mathbb Z}

\newcommand{\Deltap}{\Delta^+}

\title{An Almost-Covering Threshold for Golomb-Ruler Difference Packings}
\author{Chaohang Ma \thanks{School of Mathematical Sciences and LPMC,Nankai University. 2310285@mail.nankai.edu.cn}
\and Xiangjie Yi \thanks{School of Mathematical Sciences and LPMC,Nankai University. 2310184@mail.nankai.edu.cn}
}
\date{}

\begin{document}
\maketitle

\begin{abstract}
For a fixed integer $t\geq 3$, consider families of $t$-mark Golomb rulers whose positive-difference sets are pairwise disjoint and contained in $[1,U]$.  Let $P_t(U)$ be the largest number of integers covered by such a family.  We determine the threshold for asymptotically complete coverage:
\[
  P_t(U)=U-o(U) \quad\Longleftrightarrow\quad 3\leq t\leq 5.
\]
The cases $t=3,4$ follow from the known existence spectra for perfect difference families.  For $t=5$, Wild's product construction, in the form recorded by Mathon and applied to perfect families of orders $121$ and $161$, gives a multiplicative semigroup of exact-covering scales; an elementary density lemma on its logarithms then supplies a scale $(1-o(1))U$ below every sufficiently large $U$.

For the converse, we give a self-contained one-frequency Fourier obstruction.  If $x_0\in(\pi,3\pi/2)$ is the first positive solution of $\tan x=x$ and
\[
  \gamma_0=-\frac{2\sin x_0}{x_0}=0.4344672564\ldots,
\]
then, for every fixed $t\geq 6$,
\[
 \liminf_{U\to\infty}\left(1-\frac{P_t(U)}{U}\right)
 \geq \frac{(t-1)\gamma_0-2}{2(t-2)}.
\]
In particular, the forced gap for six-mark rulers is at least $2.1542035\%$.  We also prove a discrete small-difference bound which yields a stronger obstruction for every $t\geq14$ and forces a gap of
\[
 \frac12-\frac1{\sqrt t}-\frac7{8t}+O(t^{-3/2})
\]
as $t\to\infty$.
\end{abstract}



\section{Introduction}

For a finite set $A=\{a_1<\cdots<a_t\}\subset\Z$, write
\[
  \Deltap(A)=\{a_j-a_i:1\leq i<j\leq t\}.
\]
The set $A$ is a \emph{$t$-mark Golomb ruler} if all these positive differences are distinct.  Put
\[
  d_t=\binom t2.
\]
For an integer $U\geq1$, let $P_t(U)$ be the maximum of
\[
  \left|\bigcup_{i=1}^{b}\Deltap(A_i)\right|
\]
over all finite families of $t$-mark Golomb rulers for which the sets $\Deltap(A_i)$ are pairwise disjoint and contained in $[1,U]$.  Equivalently, every admissible family with $b$ rulers covers exactly $bd_t$ integers.  We use
\[
  \lambda_t(U)=1-\frac{P_t(U)}U
\]
for the uncovered proportion.

The problem is a packing version of a perfect difference family and, after normalizing the least mark of every ruler to zero, an admissible family of $b$ rulers is a $(b,t-1)$ difference triangle set of scope at most $U$; see, for example, \cite{CheeColbourn,Shearer}.  Exact coverage of $[1,U]$ is much more rigid than a general difference triangle set.  Indeed, a $(v,t,1)$ perfect difference family (PDF) consists of $t$-subsets whose positive differences partition $[1,(v-1)/2]$; hence $v=2U+1$ in the exact-covering situation.

Our first result locates the sharp transition between almost complete coverage and a positive-density leave.

\begin{theorem}[Almost-covering threshold]\label{thm:threshold}
For every fixed integer $t\geq3$,
\[
  P_t(U)=U-o(U) \quad (U\to\infty)
\]
if and only if $3\leq t\leq5$.
\end{theorem}

The positive half is proved in Section~\ref{sec:positive}.  The cases $t=3,4$ use exact existence theorems and in fact leave only $O(1)$ integers.  The five-mark case uses two multiplicatively independent exact-covering seeds.  Section~\ref{sec:fourier} proves the negative half by evaluating a nonnegative trigonometric polynomial at the first negative minimum of the sinc function.  Section~\ref{sec:small-difference} gives a second, elementary obstruction that is asymptotically stronger for large $t$.  Section~\ref{sec:dts} records the precise relation with published difference-triangle-set bounds.

\section{Perfect families and the positive cases}\label{sec:positive}

A $(v,t,1)$-PDF has $(v-1)/(t(t-1))$ base blocks and its positive differences partition $[1,(v-1)/2]$.  Consequently, if $v=t(t-1)m+1$, it gives $m$ rulers that cover $[1,d_tm]$ exactly.

The existence spectrum for $(v,3,1)$-PDFs is classical: such a family exists exactly when $v\equiv1,7\pmod{24}$ \cite{Skolem,BermondBrouwerGerma,KotzigTurgeon}.  The recently completed spectrum for block size four states that a $(v,4,1)$-PDF exists for every $v\equiv1\pmod{12}$, $v\geq13$, except $v=25,37$ \cite{LiuEtAl}.

\begin{proposition}\label{prop:three-four}
As $U\to\infty$,
\[
  P_3(U)=U-O(1)
  \qquad\text{and}\qquad
  P_4(U)=U-O(1).
\]
More precisely, $P_3(U)\geq U-8$ and $P_4(U)\geq U-5$ for all sufficiently large $U$.
\end{proposition}

\begin{proof}
For $t=3$, write $v=6m+1$.  The congruences $v\equiv1,7\pmod{24}$ are equivalent to $m\equiv0,1\pmod4$.  Choose the largest such $m$ with $3m\leq U$.  Consecutive admissible values of $m$ differ by at most three, and a direct check of the two missing residue classes gives $U-3m\leq8$.  The corresponding PDF partitions $[1,3m]$.

For $t=4$, take $m=\lfloor U/6\rfloor$.  Once $m\geq4$, the $(12m+1,4,1)$-PDF supplied by \cite{LiuEtAl} partitions $[1,6m]$, and $U-6m\leq5$.
\end{proof}

For five marks, write $\mathcal D_5(m)$ for a $(20m+1,5,1)$-PDF.  Thus $\mathcal D_5(m)$ has $m$ blocks and partitions $[1,10m]$.  Wild's product construction, in the form recorded by Mathon, implies \cite{Wild,Mathon}
\begin{equation}\label{eq:mathon-recursion}
  \mathcal D_5(s),\ \mathcal D_5(t)
  \quad\Longrightarrow\quad
  \mathcal D_5(20st+s+t).
\end{equation}
In terms of $v=20m+1$, this is simply multiplication:
\[
  20(20st+s+t)+1=(20s+1)(20t+1).
\]
Mathon also records $\mathcal D_5(6)$ and $\mathcal D_5(8)$, whose corresponding values of $v$ are $121$ and $161$.

We need the following elementary approximation fact.

\begin{lemma}\label{lem:semigroup}
Let $x,y>0$ and suppose that $x/y\notin\mathbb Q$.  Then
\[
 T-\max\{ax+by\leq T:a,b\in\Z_{\geq0}\}\longrightarrow0
 \qquad(T\to\infty).
\]
\end{lemma}

\begin{proof}
Fix $\varepsilon>0$.  The forward orbit $\{ax\bmod y:a\geq0\}$ is dense in the circle $\mathbb R/y\mathbb Z$.  Hence the open arcs
\[
  ax+(0,\varepsilon)\pmod y
\]
cover that circle; compactness gives a finite subcover.  Let $A_\varepsilon$ be the finite set of indices used.  For sufficiently large $T$, choose $a\in A_\varepsilon$ such that the residue
\[
  \rho=(T-ax)\bmod y
\]
belongs to $(0,\varepsilon)$.  Then
\[
 b=\frac{T-ax-\rho}{y}
\]
is a nonnegative integer and $0\leq T-(ax+by)=\rho<\varepsilon$.  Since $\varepsilon$ was arbitrary, the assertion follows.
\end{proof}

\begin{proposition}\label{prop:five}
One has
\[
  P_5(U)=U-o(U).
\]
\end{proposition}

\begin{proof}
Iterating \eqref{eq:mathon-recursion} from the two seeds shows that, for every $a,b\in\Z_{\geq0}$ not both zero, there is a perfect family whose positive differences partition
\[
  \left[1,H_{a,b}\right],
  \qquad
  H_{a,b}=\frac{121^a161^b-1}{2}.
\]
Moreover,
\[
  \frac{\log121}{\log161}\notin\mathbb Q,
\]
because a rational relation would give $121^q=161^p$ for some positive integers $p,q$, contradicting unique factorization.

Apply Lemma~\ref{lem:semigroup} with $x=\log121$, $y=\log161$, and $T=\log(2U+1)$.  We obtain $a,b\geq0$ such that
\[
  121^a161^b\leq2U+1
  \qquad\text{and}\qquad
  \frac{121^a161^b}{2U+1}=1-o(1).
\]
Thus $H_{a,b}\leq U$ and $H_{a,b}/U\to1$.  The exact covering of $[1,H_{a,b}]$ gives the desired lower bound, while $P_5(U)\leq U$ is immediate.
\end{proof}

\section{A one-frequency Fourier obstruction}\label{sec:fourier}

Let $A_1,\dots,A_b$ be an admissible family and translate every ruler so that its least mark is zero.  Put
\[
  D=\bigsqcup_{i=1}^b\Deltap(A_i),
  \qquad
  L=[1,U]\setminus D,
  \qquad
  \ell=|L|.
\]
Since each ruler contributes $d_t$ differences,
\begin{equation}\label{eq:count}
  bd_t=U-\ell.
\end{equation}
For $P_i(z)=\sum_{a\in A_i}z^a$ and every real $\theta$,
\begin{align}
0
&\leq \sum_{i=1}^b\left|P_i(e^{\mathrm i\theta})\right|^2 \\
&=bt+2\sum_{d\in D}\cos(d\theta). \label{eq:positive-poly}
\end{align}
Introduce the Dirichlet kernel
\[
  D_U(\theta)=1+2\sum_{d=1}^U\cos(d\theta)
  =\frac{\sin((U+\tfrac12)\theta)}{\sin(\theta/2)}.
\]
Using $D=[1,U]\setminus L$ in \eqref{eq:positive-poly} gives
\[
0\leq bt-1+D_U(\theta)-2\sum_{d\in L}\cos(d\theta)
\leq bt-1+D_U(\theta)+2\ell.
\]
We have therefore proved the necessary condition
\begin{equation}\label{eq:fourier-necessary}
  0\leq bt-1+D_U(\theta)+2\ell
  \qquad(\theta\in\mathbb R).
\end{equation}

Let $x_0\in(\pi,3\pi/2)$ be the first positive root of $\tan x=x$, and set
\[
  \gamma_0=-\frac{2\sin x_0}{x_0}
  =0.434467256422\ldots.
\]

\begin{theorem}\label{thm:fourier}
For every fixed $t\geq6$,
\[
 \liminf_{U\to\infty}\lambda_t(U)
 \geq c_t^{\mathrm F},
 \qquad
 c_t^{\mathrm F}=\frac{(t-1)\gamma_0-2}{2(t-2)}.
\]
The right-hand side is positive for every integer $t\geq6$.
\end{theorem}

\begin{proof}
In \eqref{eq:fourier-necessary}, take
\[
  \theta=\frac{x_0}{U+1/2}.
\]
Then
\begin{align*}
 D_U(\theta)
 &=\frac{\sin x_0}{\sin(x_0/(2U+1))}
 =-\gamma_0\left(U+\frac12\right)+O(U^{-1}).
\end{align*}
Also, by \eqref{eq:count},
\[
  bt=\frac{2(U-\ell)}{t-1}.
\]
Substitution into \eqref{eq:fourier-necessary}, before taking any asymptotic limit, yields
\begin{equation}\label{eq:finite-fourier}
  \ell\geq
  \frac{(t-1)\bigl(1-D_U(\theta)\bigr)-2U}{2(t-2)}.
\end{equation}
Consequently,
\[
  \frac\ell U\geq
  \frac{(t-1)\gamma_0-2}{2(t-2)}+O(U^{-1}).
\]
This holds for every admissible family, hence for an optimal one.  Finally,
$(t-1)\gamma_0>2$ first occurs at the integer $t=6$.
\end{proof}

For $t=6$, Theorem~\ref{thm:fourier} gives
\[
  \liminf_{U\to\infty}\lambda_6(U)
  \geq0.021542035264\ldots.
\]
Together with Propositions~\ref{prop:three-four} and \ref{prop:five}, this proves Theorem~\ref{thm:threshold}.

\section{A discrete small-difference bound}\label{sec:small-difference}

The preceding Fourier argument has a limiting leave of $\gamma_0/2\approx21.72\%$ as $t\to\infty$.  A direct sum-of-differences argument is stronger for large $t$.

Normalize each ruler as
\[
  A_i=\{0=a_{i,0}<a_{i,1}<\cdots<a_{i,t-1}\leq U\}.
\]
Fix an integer $q$ with $1\leq q\leq t-1$, and in every ruler select all differences between marks at index distance at most $q$:
\[
  a_{i,r+s}-a_{i,r},
  \qquad 1\leq s\leq q,\quad 0\leq r\leq t-1-s.
\]
There are
\[
  m_q=\sum_{s=1}^q(t-s)
  =q\left(t-\frac{q+1}{2}\right)
\]
such differences per ruler.

\begin{theorem}\label{thm:small-diff}
Define
\[
  A_t^{\mathrm S}
  =\max_{1\leq q\leq t-1}
  \frac{q}{q+1}\left(t-\frac{q+1}{2}\right)^2.
\]
Then
\[
  \limsup_{U\to\infty}\frac{P_t(U)}U
  \leq \min\left\{1,\frac{d_t}{A_t^{\mathrm S}}\right\}.
\]
Equivalently,
\[
  \liminf_{U\to\infty}\lambda_t(U)
  \geq c_t^{\mathrm S},
  \qquad
  c_t^{\mathrm S}=\max\left\{0,1-\frac{d_t}{A_t^{\mathrm S}}\right\}.
\]
\end{theorem}

\begin{proof}
Let the family contain $b$ rulers.  The $bm_q$ selected differences are distinct positive integers, so their sum is at least
\[
  \frac{bm_q(bm_q+1)}2.
\]
For a fixed ruler and a fixed $s$, each consecutive gap
$a_{i,j}-a_{i,j-1}$ occurs in at most $s$ of the sums
$a_{i,r+s}-a_{i,r}$.  Hence
\[
  \sum_{r=0}^{t-1-s}(a_{i,r+s}-a_{i,r})\leq sU.
\]
Summing over $s\leq q$ and over the $b$ rulers gives the upper bound
\[
  bU\frac{q(q+1)}2.
\]
Consequently,
\[
  bm_q(bm_q+1)\leq bUq(q+1).
\]
Writing $h_q=t-(q+1)/2$, so that $m_q=qh_q$, and discarding only a nonnegative lower-order term, we obtain
\[
  \frac bU\leq\frac{q+1}{qh_q^2}.
\]
Since the covered proportion is $bd_t/U$, optimizing over $q$ proves the result.
\end{proof}

This is a direct version of Kløve's small-difference argument
for difference triangle sets \cite{Kløve}; compare Riblet's
related bounds for Sidon sets in unions of intervals \cite{Riblet}.  Presenting it directly avoids any need to choose translations of the rulers or to control cross-interval differences.

\begin{corollary}\label{cor:asymptotic}
As $t\to\infty$,
\[
  c_t^{\mathrm S}
  =\frac12-\frac1{\sqrt t}-\frac7{8t}+O(t^{-3/2}).
\]
Moreover,
\[
  c_t^{\mathrm S}>c_t^{\mathrm F}
  \qquad\text{for every integer }t\geq14.
\]
\end{corollary}

\begin{proof}
For a real variable \(q\in[1,t-1]\), put
\[
  f_t(q)=\frac{q}{q+1}\left(t-\frac{q+1}{2}\right)^2.
\]
On this interval, \(f_t\) has the unique maximizer
\[
  q_* =\frac{\sqrt{16t+1}-3}{4}
  =\sqrt t-\frac34+O(t^{-1/2}).
\]
Choosing a nearest integer changes $f_t(q_*)$ by a relative $O(t^{-3/2})$.  If $s=q_*+1$, then the critical-point equation is $t=s^2-s/2$ and
\[
  f_t(q_*)=s(s-1)^3.
\]
Expanding $1-d_t/f_t(q_*)$ gives
\[
  \frac12-\frac1{\sqrt t}-\frac7{8t}+O(t^{-3/2}),
\]
and the rounding error is absorbed in the remainder.

For the comparison, the choice $q=3$ gives
\[
  c_t^{\mathrm S}\geq1-\frac{2t(t-1)}{3(t-2)^2}.
\]
This lower bound exceeds $c_t^{\mathrm F}$ exactly when
\[
 (2-3\gamma_0)t^2+(-14+9\gamma_0)t+(12-6\gamma_0)>0.
\]
The left-hand side factors as
\[
  (t-1)\bigl((2-3\gamma_0)t-(12-6\gamma_0)\bigr).
\]
Since $\gamma_0=0.434467\ldots<4/9$, its second factor is positive at $t=14$ and increases with $t$.
\end{proof}

For orientation, some values are listed below.
\begin{center}
\begin{tabular}{c@{\qquad}cc@{\qquad}c@{\qquad}cc}
\toprule
$t$ & $c_t^{\mathrm F}$ & $c_t^{\mathrm S}$ & $t$ & $c_t^{\mathrm F}$ & $c_t^{\mathrm S}$\\
\midrule
6  & 0.021542 & 0        & 12 & 0.138957 & 0.120000\\
7  & 0.060680 & 0        & 13 & 0.146073 & 0.140496\\
8  & 0.086773 & 0.005917 & 14 & 0.152003 & 0.157407\\
10 & 0.119388 & 0.065744 & 20 & 0.173747 & 0.224490\\
\bottomrule
\end{tabular}
\end{center}

\section{Relation with difference triangle sets}\label{sec:dts}

Let $M(I,J)$ denote the minimum scope of an $(I,J)$ difference triangle set in the convention of \cite{Shearer}: there are $I$ rulers and $J+1$ marks per ruler.  An admissible family in our problem with $b$ rulers is a $(b,t-1)$ difference triangle set.  Thus any asymptotic lower bound
\[
  M(I,t-1)\geq (a_t+o(1))I
\]
implies
\[
  \limsup_{U\to\infty}\frac{P_t(U)}U\leq\frac{d_t}{a_t}.
\]

Shearer proved the closed-form bounds \cite[Remark~2]{Shearer}
\[
  M(I,5)\geq\frac{91I+6}{6},
  \qquad
  M(I,6)\geq\frac{179I+9}{8}.
\]
They imply, respectively,
\[
  \liminf_{U\to\infty}\lambda_6(U)\geq\frac1{91}
  \quad\text{and}\quad
  \liminf_{U\to\infty}\lambda_7(U)\geq\frac{11}{179}.
\]
The Fourier constant improves Shearer’s displayed six-mark bound,
\[
  0.021542035\ldots>\frac1{91}=0.010989010\ldots,
\]
whereas Shearer's seven-mark value is slightly stronger:
\[
  \frac{11}{179}=0.061452514\ldots>0.060680353\ldots.
\]

The general linear-programming framework of Lorentzen--Nilsen and Shearer \cite{LorentzenNilsen,Shearer} contains more information than either one of the explicit bounds in Sections~\ref{sec:fourier} and \ref{sec:small-difference}.  Numerical comparisons for further values of $t$ require either reproducible exact-arithmetic computations or dual certificates; no unverified numerical LP claims are needed for Theorem~\ref{thm:threshold}.

\section{Concluding remarks}

The threshold at five marks comes from two different mechanisms.  On the positive side, exact five-mark families occur at enough multiplicatively generated scales to approximate every large scope from below.  On the negative side, the Dirichlet kernel detects a density obstruction as soon as $t=6$.  Thus the classical nonexistence of exact PDFs for $t\geq6$ \cite{BermondKotzigTurgeon} is strengthened here from ``at least one missing difference'' to a leave of positive density.

The large-$t$ small-difference bound is elementary and asymptotically forces half of the interval to remain uncovered.  It would be interesting to compare it analytically, for all $t$, with the limiting Shearer linear program, or to exhibit a family of exact dual certificates.  A separate question is whether the constants in the six- and seven-mark cases can be improved by using several Fourier frequencies simultaneously.

\section*{Declaration on the use of generative AI.}
Generative AI tools, including ChatGPT, were used as auxiliary tools for language polishing, improving the presentation, and checking the clarity and consistency of the manuscript. The research problem, overall strategy, mathematical ideas, constructions, proofs, and conclusions were conceived and developed by the authors. All AI-assisted suggestions were independently examined and, where adopted, revised and verified by the authors, who take full responsibility for the content of this work.

\end{document}